\documentclass[11pt,reqno,a4paper]{amsart}
\usepackage{hyperref}
\usepackage{amsmath}
\usepackage{amssymb,a4wide}
\usepackage{amsthm}
\usepackage[margin=3cm]{geometry}
\usepackage{graphicx,epstopdf}

\hypersetup{
    colorlinks=true,
    linkcolor=black,
    filecolor=black,      
    urlcolor=blue,
    citecolor=red,
}

\newcommand{\guio}[1]{\nobreakdash-\hspace{0pt}}

\numberwithin{equation}{section}
\newtheorem{theorem}{Theorem}[section]
\newtheorem{lemma}[theorem]{Lemma}
\newtheorem{prop}[theorem]{Proposition}

\theoremstyle{definition}

\newtheorem{assumption}{Assumption}

\newtheorem{remark}[theorem]{Remark}

\newcommand{\R}{\mathbb{R}}

\newcommand{\Per}{\operatorname{Per}}

\newcommand{\supp}{\operatorname{supp}}

\title[Shape Optimisation for nonlocal anisotropic energies]
{Shape Optimisation for nonlocal anisotropic energies}
\author[R.~Cristoferi]{R.~Cristoferi}
\author[M.G.~Mora]{M.G.~Mora}
\author[L.~Scardia]{L.~Scardia}
\address[R. Cristoferi]{Mathematics Department, Institute for Mathematics, Astrophysics, and Particle Physics, Radboud University, The Netherlands}
\email{riccardo.cristoferi@ru.nl}

\address[M.G. Mora]{Dipartimento di Matematica, Universit\`a di Pavia, Italy}
\email{mariagiovanna.mora@unipv.it}

\address[L. Scardia]{Department of Mathematics, Heriot-Watt University, United Kingdom}
\email{L.Scardia@hw.ac.uk}

\date{}

\begin{document}

\begin{abstract}
In this paper we consider shape optimisation problems for sets of prescribed mass, where the driving energy functional is nonlocal and \textit{anisotropic}. More precisely, we deal with the case of attractive/repulsive interactions in two and three dimensions, where the attraction is quadratic and the repulsion is given by an anisotropic variant of the Coulomb potential. 
%We show that the minimising shapes depend on two crucial quantities: the fixed value of the mass, and the sign of the Fourier transform of the interaction potential. 
Under the sole assumption of strict positivity of the Fourier transform of the interaction potential, we show the existence of a threshold value for the mass above which the minimiser is an ellipsoid, and below which the minimiser does not exist. If, instead, the Fourier transform of the interaction potential is only nonnegative, we show the emergence of a dichotomy: either there exists a threshold value for the mass as in the case above, or the minimiser is an ellipsoid for any positive value of the mass. 
\bigskip


\medskip

\end{abstract}

\maketitle

\section{Introduction}
In this paper we consider a class of nonlocal and \textit{anisotropic} shape optimisation problems for sets of prescribed mass.
More precisely, for a given mass $m>0$, we are interested in the minimisation of the energy functional 
\begin{equation}\label{en:general-dd}
\mathcal I(\Omega)= \int_{\Omega}\int_{\Omega} \bigg(W(x-y)+\frac12|x-y|^2\bigg) \,d x dy,
\end{equation}
over the class of sets with mass $m$, 
$$
\mathcal{A}_m= \big\{ \Omega\subset \R^d: \ \Omega \, \text{ measurable, } |\Omega|=m\big\},
$$
for $d=2,3$. In \eqref{en:general-dd}, the interaction potential $W$ is defined for $x\neq 0$ as 
\begin{equation}\label{potdef-dd}
W(x)= 
\begin{cases}
\medskip
\displaystyle -\log|x| + \kappa\bigg(\frac{x}{|x|}\bigg) \quad &\text{if } d=2,\\
\displaystyle\frac1{|x|}\kappa\bigg(\frac{x}{|x|}\bigg) &\text{if } d=3,
\end{cases}
\end{equation}
and $W(0)=+\infty$, and satisfies the following assumption.

\begin{assumption}\label{assumption-k} We require the profile $\kappa: \mathbb{S}^{d-1}\to \R$ to be even, and such that both $W$ and $\widehat W$ are continuous on $\mathbb{S}^{d-1}$. Additionally, if $d=3$ we require $\kappa$ to be strictly positive on $\mathbb{S}^{d-1}$.
\end{assumption}

The energy $\mathcal I$ is the sum of two competing terms: an attractive, quadratic interaction, that dominates at large distances, and a repulsive, Coulomb-like interaction, driven by the anisotropic potential $W$. The additional positivity requirement for $\kappa$ in the three-dimensional case is there to preserve the repulsive nature of $W$; this is not needed for $d=2$ since $\kappa$ is bounded, and hence at short range the repulsive nature of $-\log|\cdot|$ is not affected by the additional anisotropy $\kappa$.

The potential $W$ is an anisotropic extension of the classical, radially symmetric Coulomb potential, which corresponds to the special case of a constant profile $\kappa$. The anisotropy is fully encoded in the profile $\kappa$, which introduces an additional dependence on the directions of interaction. The apparent mismatch between the additive structure of the potential in the two-dimensional case and the multiplicative structure for $d=3$ is due to the logarithm being the limit of Riesz potentials. Indeed, the additive radial/angle splitting in two dimensions is a `limit' multiplicative splitting for Riesz, since
$$
|x|^{-s}\bigg(\frac1s+ \kappa\bigg(\dfrac{x}{|x|}\bigg) \bigg) - \frac1s \to -\log|x| + \kappa\bigg(\frac{x}{|x|}\bigg) \quad \text{as } s\to 0^+.
$$
Moreover, this apparent mismatch disappears in Fourier space: the structure of $\widehat{W}$ for $d=2$ and $d=3$ is identical (see \eqref{FTW-dd}) and exhibits a multiplicative splitting between its radial and directional components.

\

The main result of this work is the characterisation of the minimisers of  $\mathcal I$ in the class of sets $\mathcal{A}_m$, for any mass $m>0$. This is done under the sole assumption that the Fourier transform $\widehat W$ of the potential $W$ on the sphere $\mathbb{S}^{d-1}$ is nonnegative. 

In fact we have two main results, depending on whether $\widehat W$ is strictly positive or not. In the first case we show that above a given threshold for the mass the unique minimiser of $\mathcal I$ is a $d$-dimensional ellipsoid. Uniqueness has to be intended up to translations, since the functional $\mathcal I$ is translation-invariant. More precisely we have the following.

\begin{theorem}\label{Mainthm}
Let $W$ be a potential defined as in \eqref{potdef-dd} and satisfying Assumption \ref{assumption-k}. Assume that $\widehat{W}> 0$ on $\mathbb{S}^{d-1}$.
Then there exists a critical value $m^\ast>0$ such that for $m<m^\ast$ the energy $\mathcal I$ in \eqref{en:general-dd} has no miminiser in $\mathcal{A}_m$, while for $m\geq m^\ast$ the unique minimiser of $\mathcal I$ in $\mathcal{A}_m$ is a $d$-dimensional ellipsoid, up to translations.  Moreover, as $m\to +\infty$, the optimal ellipsoids converge to a ball, namely the ratios of the semi-axes converge to one.
\end{theorem}

In the case of degeneracy of $\widehat W$, instead, we have the following dichotomy.

\begin{theorem}\label{Mainthm2}
Let $W$ be a potential defined as in \eqref{potdef-dd} and satisfying Assumption \ref{assumption-k}. Assume that $\widehat{W}\geq 0$ on $\mathbb{S}^{d-1}$.
Then we are either in the case of Theorem~\ref{Mainthm}, or the unique minimiser of $\mathcal I$ in $\mathcal{A}_m$ is a $d$-dimensional ellipsoid, up to translations, for every $m>0$. Moreover, as $m\to +\infty$, the optimal ellipsoids converge to a ball.
\end{theorem}

The occurrence of one or the other possibility in Theorem~\ref{Mainthm2} is related to the minimisation problem for the energy \eqref{en:general-dd} in the wider class of measures (rather than sets) with prescribed mass, which we briefly recall.

In \cite{CS2d,CS3d} and \cite{MMRSV-3d, Mora-Muenster}, the authors consider the same anisotropic functional as in \eqref{en:general-dd}, but defined on the set $\mathcal{P}(\R^d)$ of
probability measures, for $d=2,3$. More precisely, the authors study the minimisation of the functional 
\begin{equation}\label{en:usual}
\mathcal{I}(\mu) = \int_{\R^d}\int_{\R^d} \bigg(W(x-y)+\frac12|x-y|^2\bigg)\,d\mu(x) d\mu(y),% \quad \mu \in \mathcal{P}(\R^d),
\end{equation}
for $\mu\in\mathcal{P}(\R^d)$, where $W$ is as in \eqref{potdef-dd}.  

The main result in \cite{CS2d,CS3d} and \cite{MMRSV-3d, Mora-Muenster} is that, under the assumption that $\widehat{W}\geq 0$ on $\mathbb{S}^{d-1}$, the minimiser $\mu^\ast$ of $\mathcal I$ in $\mathcal{P}(\R^d)$ is unique up to translations and can be characterised as follows:
\begin{enumerate}%[\textnormal{(}a\textnormal{)}]
\item\label{aaa} If $\widehat{W}> 0$ on $\mathbb{S}^{d-1}$,
then $\mu^\ast$ is of the form
\begin{equation}\label{ellipsoidlaw-intro}
\mu^*=\frac{\chi_{E^*}}{|E^*|},
\end{equation}
for a $d$-dimensional ellipsoid $E^\ast\subset \R^d$ centred at $0$.
\item If $\widehat{W}\geq 0$ on $\mathbb{S}^{d-1}$, then
either $\mu^*$ is as in \eqref{ellipsoidlaw-intro} or $\mu^*$ is supported on a $(d-1)$-dimensional set - an ellipse if $d=3$, and a segment if $d=2$. In the case of loss of dimensionality the density of the measure $\mu^*$ is not constant, and in particular it is given by the celebrated semi-circle law in the two-dimensional case.
\end{enumerate}

Since $\mathcal{I}(\mu)$ is a quadratic function of $\mu$, optimal measures - or more precisely their shapes - do not change by changing the mass; in other words, 
for any $m>0$ the energy $\mathcal I$ is minimised in the class $m\mathcal{P}(\R^d)$ simply by $m\mu^*$.

The situation is different for sets, since the attractive and the repulsive part of the energy rescale differently under a dilation. To see this, let $\Omega\in \mathcal{A}_{m}$; we set $\Omega_{0}:=\left({1}/{m}\right)^{1/d}\Omega$, so that $\Omega_{0}\in \mathcal{A}_{1}$. By changing variables we have that 
\begin{align*}
\mathcal{I}(\Omega) &= \mathcal{I}(m^{1/d}\,\Omega_{0})\\
&
=m^{\frac{d+2}d} 
\bigg(
\int_{\Omega_{0}}\int_{\Omega_{0}} \bigg(W(x-y)+\frac{m}2|x-y|^2\bigg) \,dxdy \bigg) - \frac{m^2}2\delta_{d2}\log m
 =: J_m(\Omega_0),
\end{align*}
where the last term activates only for $d=2$, and is due to the logarithm not being fully zero-homogeneous. Hence, minimising $\mathcal I$ on $\mathcal{A}_{m}$ is equivalent to minimising $J_m$ on $\mathcal{A}_{1}$. This shows that if $m\ll 1$ the attraction term will weigh comparatively less than the nonlocal repulsive term, and viceversa if $m\gg 1$. In other words, for large mass attraction dominates, while for small mass repulsion dominates.

As a corollary of this observation, for large mass it is natural to expect the optimal shapes to be more and more `isotropic', since the dominating attraction term is radially symmetric.

\

Assume now that the Fourier transform of $W$ is strictly positive, so that $\mathcal{I}$ is minimised on $\mathcal{P}(\R^d)$
by the measure $\mu^*$ in \eqref{ellipsoidlaw-intro}. 
The critical threshold $m^*$ for the mass in Theorem~\ref{Mainthm} is exactly the mass of the support $E^*$ of $\mu^*$.

In the special case $m=m^*$ it is clear that the unique minimiser of $\mathcal I$ in the class of sets $\mathcal{A}_{m^*}$ is exactly $E^*$, up to translations, since $\chi_\Omega\in m^*\mathcal{P}(\R^d)$ for every $\Omega\in\mathcal{A}_{m^*}$. However, for $m\neq m^*$ the connection between the two problems is less clear. Indeed, the minimiser $m\mu^*$ in the class $m\mathcal{P}(\R^d)$ is never a set, and so for the minimisation problem on sets a new approach has to be developed. This is the main contribution of this paper.

Theorem~\ref{Mainthm} ensures that for $m>m^*$ the minimiser on sets is still an ellipsoid $E$. We note, however, that its relation with the ellipsoid $E^*$ at the critical mass is not straightforward. In particular, $E$ is not simply given by the dilated set $(m/m^*)E^*$, since by Theorem~\ref{Mainthm} it becomes more and more rounded as $m$ gets larger.

For a degenerate $\widehat W$ the dichotomy in our second main result, Theorem~\ref{Mainthm2}, exactly corresponds to the occurrence of a full or lower dimensional minimiser for the problem on measures.
More precisely, if the minimiser on $\mathcal{P}(\R^d)$ is the measure $\mu^*$ in \eqref{ellipsoidlaw-intro}, then the problem on sets exhibits a critical mass $m^*$ given by $|E^*|$.
If, instead, the minimiser on $\mathcal{P}(\R^d)$ is a measure with a $(d-1)$-dimensional support, then the problem on sets has a solution for every mass $m>0$.

\subsection{Motivation and comparison with the radially symmetric case} The problem we consider can be interpreted as a first shape optimisation result for nonlocal anisotropic energies with competing attractive and repulsive terms. 

The isotropic counterpart of this problem is well-studied and its analysis started in the seminal work \cite{BCT18}, where the authors consider the energy 
\begin{equation}\label{en:E}
{\mathcal E}(\Omega)=  \int_{\Omega}\int_{\Omega} K(x-y)\,dx dy, \quad \Omega\in \mathcal{A}_m,
\end{equation}
with $K$ a power-law potential of the form 
\begin{equation}\label{K-pq}
K(x)=\frac{|x|^q}{q}-\frac{|x|^p}{p}, \quad -d<p<0, \quad q>0.
\end{equation}
The sum above combines an attractive term (the $q$-power), dominant at long ranges, with a repulsive term (the $p$-power) that dominates the interactions at short range. In  \cite{BCT18}, the majority of the results pertain to quadratic attraction $q=2$ and to repulsion with $p\leq 2-d$, where $d-2$ is the critical exponent of the Coulomb potential (with the understanding that $-\frac{|x|^p}{p} =- \log|x|$ for $d=2$ and $p=2-d=0$). In particular, for $q=2$ and $p=2-d$, the energy \eqref{en:E} becomes
\begin{equation}\label{en:EC}
{\mathcal E}(\Omega) = \int_{\Omega}\int_{\Omega} \bigg(W_{\text{C}}(x-y)+\frac12|x-y|^2\bigg)\, dx dy, \quad \Omega\in \mathcal{A}_m,
\end{equation}
where we denoted with $W_{\text{C}}$ the isotropic Coulomb potential. In other words, it coincides with \eqref{en:general-dd} in the case of a constant profile $\kappa$.
Moreover, the energy \eqref{en:EC} can be considered as a toy model for the energy of spring-like media, for which the short-range interactions are Coulombic and repulsive, while at long range a Hookean attraction dominates.

In \cite{BCT18} the authors show that there is a threshold for the mass, given by the volume of the unit ball $B$, such that the energy $\mathcal E$ in \eqref{en:EC} admits no minimiser if $m<|B|$, while for $m\geq |B|$ the minimisers of $\mathcal E$ are the balls of mass $m$. This is consistent with Theorem~\ref{Mainthm}, since one can show that $\frac1{|B|}\chi_B$ is exactly the minimiser (up to translations) of the energy $\mathcal E$ in \eqref{en:EC} among probability measures.

For more general power-law interactions \eqref{K-pq} with some restrictions on $p$ and quadratic attraction $q=2$ 
the authors in \cite{BCT18} prove the existence of two thresholds for the mass, $0<m_1\leq m_2$, such that 
the problem has no minimisers if $m<m_1$, whereas balls are the unique minimisers if $m>m_2$.
These results have been extended to a general power-law attraction $q>0$ in \cite{FL18, FL21} and to a more general but still isotropic repulsive kernel in \cite{Car}. It is not known whether the two mass thresholds coincide, except in the exact case \eqref{en:EC}.

Note that in \eqref{K-pq} (as in \eqref{potdef-dd}) the repulsion is \textit{strong}: in particular the relaxed energy on measures does not allow concentration on points, and atoms have infinite energy. The case of weak repulsion has been considered in \cite{CPT-Preprint}, the model case being that of power-law potentials of the form \eqref{K-pq}, with $q>p>0$ (and hence with $K(0)=0$). For these energies, as for the energies in \cite{FL21}, the optimal measure is generally not expected to be a set. However, if the optimal measure concentrates on the vertices of a regular $(d+1)$-gon, or on a sphere, then positive results for the minimisation on sets can be inferred.

Another important class of isotropic attractive/repulsive energies is given by $$
{\mathcal E}_{\text P}(\Omega) = \int_{\Omega}\int_{\Omega} W_{\text{C}}(x-y) \,dx  dy +\Per(\Omega), \quad \Omega \in \mathcal{A}_m,
$$
where $\Per$ denotes the classical perimeter. Characterising minimisers of ${\mathcal E}_{\text P}$ corresponds to finding solutions to a nonlocal isoperimetric problem. The energies ${\mathcal E}_{\text P}$ have been first introduced by Gamow in his liquid drop model and widely studied since. What is known so far is that there exist two (possibly different) thresholds for the mass, $0<m_1 \leq m_2$, such that balls are the unique minimisers for $m\leq m_1$, while minimising sets fail to exist for $m> m_2$ (see \cite{CR, FE15, Julin17, KM14, LO14}). 
Only for certain Riesz repulsive potentials, it was proved in \cite{BC14} that $m_1=m_2$.
The main difference with \eqref{en:EC} is that the power-law attraction is replaced by the perimeter. While the power-law attraction dominates for large mass, the perimeter dominates for small mass, so the thresholds for the minimality of balls are reversed. Although the results for the nonlocal isoperimetric problem and for the attractive/repulsive power-law interactions are similar, the nature of the two problems is very different and so are the techniques of proof. Considering an anisotropic analogue of ${\mathcal E}_{\text P}$ is a very natural direction of investigation, but this is not a direction we will pursue here.

\subsection{Method of proof} 
One of the main difficulties in proving existence of an optimal set for the energy $\mathcal I$ is establishing compactness for minimising sequences in $\mathcal A_m$.
Indeed, a minimising sequence of sets in $\mathcal A_m$ will in general converge weakly to a density function taking values in $[0,1]$.
The strategy we follow consists in proving directly that a specific candidate - an ellipsoid - is a solution of the problem.
While this is the same approach used in \cite{BCT18} for the isotropic functional \eqref{en:EC}, the corresponding proofs are substantially different. 
Indeed, in the radial case there is an obvious candidate, the ball, which is completely determined once the mass is prescribed.
In the anisotropic setting the identification of a natural candidate, even among ellipsoids, is not trivial, since the mass constraint is clearly not enough to determine the semi-axes and the possible rotation of the ellipsoid with respect to the coordinate axes.

For the proof of existence, we consider the relaxed energy 
\begin{equation}\label{en:general-rho}
{\mathcal I}(\rho)= \int_{\R^d}\int_{\R^d} \bigg(W(x-y)+\frac12|x-y|^2\bigg) \rho(x)\rho(y)\,d x dy,
\end{equation}
which extends \eqref{en:general-dd} to the class of densities 
\begin{equation}\label{def:Am1}
\mathcal{A}_{m,1}=\left\{\rho \in L^1(\R^d)\cap L^\infty(\R^d): \ \|\rho\|_{L^1}=m, \ 0\leq \rho\leq 1 \text{ a.e.}
\right\}.
\end{equation}
It was proved in \cite{BCT18} that a set $\Omega\in \mathcal{A}_m$ is a minimiser of \eqref{en:general-dd} if and only if its characteristic function $\chi_\Omega \in \mathcal{A}_{m,1}$ is a minimiser of the relaxed energy \eqref{en:general-rho}, and the same holds true in our case. Since for small mass the minimising densities are not the characteristic functions of a set (see Section~\ref{sect:sub}), our original problem on sets can only have a solution for large enough mass. 

For large mass we then (equivalently) study the problem on densities, for which existence and compact support of minimisers can be proved by standard arguments. 

A special feature of the quadratic attraction potential is that in the subclass $\mathcal{A}^0_{m,1}\subset \mathcal{A}_{m,1}$ of densities with barycentre at the origin (and hence zero first moments), the energy \eqref{en:general-rho} rewrites as 
\begin{equation}\label{en:general-dd2}
{\mathcal I}(\rho)=I_m(\rho):= \int_{\R^d}\int_{\R^d} W(x-y) \rho(x)\rho(y)\,d x dy +m\int_{\R^d}|x|^2 \rho(x)\,dx, \quad \rho \in \mathcal{A}^0_{m,1}.
\end{equation}
Uniqueness, up to translations, follows by the sign condition on the Fourier transform of $W$, which implies that the energy $I_m$ is strictly convex. Strict convexity of the energy, in its turn, guarantees that  the minimiser can be characterised as the only solution of the Euler-Lagrange optimality conditions. 
Motivated by the results in \cite{CS2d,CS3d} and \cite{MMRSV-3d, Mora-Muenster} we look for a candidate ellipsoid $E\subset \R^d$ centred at the origin, with $|E|=m$, such that its characteristic function $\chi_E$ satisfies the Euler-Lagrange conditions for $I_m$, namely
\begin{align}\label{1}
(W\ast \chi_E) (x) + m\frac{|x|^2}2&= \lambda \qquad \text{if } x\in \partial E,\\ \label{2}
\smallskip
(W\ast \chi_E) (x) + m\frac{|x|^2}2&\le \lambda \qquad \text{if } x\in E^\circ,\\ 
\smallskip
(W\ast \chi_E) (x) + m\frac{|x|^2}2&\ge \lambda \qquad \text{if } x\in \R^d\setminus E, \label{3}
\end{align}
for a constant $\lambda\in \R$. To evaluate the potential of a generic ellipsoid $E$ we use the representation of the potential in Fourier form proved in \cite{MMRSV-3d, Mora-Muenster} for $d=2,3$. Following \cite{MMRSV-3d, Mora-Muenster} one can see that condition \eqref{3} is automatically satisfied by any solution $E$ of \eqref{1}--\eqref{2}. 

The key idea to solve \eqref{1}--\eqref{2} is to rewrite \eqref{1} as the stationarity condition for an auxiliary scalar function $f$ defined on symmetric and positive definite matrices $M$ (encoding the information on the semi-axes and orientation of $E$), under the determinant constraint $\det M=\frac{m^2}{|B|^2}$ (encoding the mass constraint $|E|=m$). One of the main advantages of this alternative formulation is that \eqref{2} corresponds to a condition on the sign of the Lagrange multiplier associated to the constraint.

The strategy is then to first show that the auxiliary minimisation problem for $f$  obtained by replacing the equality constraint for the determinant with the unilateral condition $\det M\geq \frac{m^2}{|B|^2}$ admits a solution. As a final step we show that this solution in fact satisfies the equality constraint. This immediately gives the required sign condition for the multiplier, and concludes the proof of  \eqref{1}--\eqref{2}.

As shown in Lemma~\ref{lemma} the auxiliary function $f$ coincides with the energy $\mathcal{I}$ evaluated on normalised characteristic functions of ellipsoids.
In particular, minimising $f$ under the unilateral constraint $\det M\geq \frac{m^2}{|B|^2}$ corresponds to minimising $\mathcal{I}$ on densities of the form $\chi_E/|E|$ for $E$ ellipsoid centred at $0$ with $|E|\geq m$. Our proof shows the surprising result that, for an ellipsoid, being a critical point of this constrained problem on ellipsoids is in fact equivalent to being a critical point of $\mathcal{I}$ on $\mathcal{A}_{m,1}$.

%In Lemma~\ref{eq:fIm} we uncover a surprising connection between the energy $I_m$ and the auxiliary function $f$: in essence, $f$ coincides with the energy computed on ellipsoids, up to a normalisation.

\subsection{Outline of the paper}
The paper is organised as follows. In Section~\ref{sect:prelim} we collect some preliminary results that will be used in the paper. In particular, in Section~\ref{sec:potential-dd} we recall the Fourier representation of $W\ast \chi_E$, computed in \cite{Mora-Muenster, MMRSV-3d}, which plays a crucial role in our approach. Section~\ref{sec:main-res} is the heart of the paper and contains the proof of our main results, Theorems \ref{Mainthm} and \ref{Mainthm2}. Finally, in Section~\ref{sect:further} we prove some further properties of the optimal ellipsoids and of the energy.

\section{Preliminary results}\label{sect:prelim}
In this section we collect some preliminary results, mostly derived from \cite{BCT18} and \cite{MMRSV-3d,Mora-Muenster}.

%\section{The energy and its relaxation}
%The energy $\mathcal I$ defined in \eqref{en:general-dd} can be extended (relaxed) to a larger class of competitors. As a first step, we extend $\mathcal I$ to the following class of functions:
%$$
%\mathcal{A}_{m,1}:=\left\{\rho \in L^1(\R^d)\cap L^\infty(\R^d):\ \|\rho\|_{L^1}=m, \ 0\leq \rho\leq 1 \text{ a.e.}
%\right\}.
%$$
%For $\rho \in \mathcal{A}_{m,1}$, the energy extends in the obvious way as 
%$$
%{\mathcal I}(\rho):= \int_{\R^d}\int_{\R^d} \bigg(W(x-y)+\frac12|x-y|^2\bigg) \rho(x)\rho(y)\,d x dy,
%$$
%As a second step, we extend the energy to the class $m\mathcal{P}(\R^d)$. For $\mu \in m\mathcal{P}(\R^d)$, the energy extends in the obvious way as 
%$$
%\mathcal{I}(\mu) := \int_{\R^d}\int_{\R^d} \bigg(W(x-y)+\frac12|x-y|^2\bigg)\,d\mu(x) d\mu(y).% \quad \mu \in \mathcal{P}(\R^d),
%$$

\subsection{Minimisation of $\mathcal I$ on $\mathcal{A}_{m,1}$ and on $\mathcal{A}_{m}$.} The following result provides existence and uniqueness (up to translations) of the minimiser of $\mathcal I$ in the relaxed class $\mathcal{A}_{m,1}$, introduced in \eqref{def:Am1}, as well as its characterisation in terms of the Euler-Lagrange conditions.
 
\begin{prop}\label{ex+uniq-d}
Let $W$ be as in \eqref{potdef-dd} with $\kappa$ satisfying Assumption \ref{assumption-k}. Assume that $\widehat{W}\geq 0$ on $\mathbb{S}^{d-1}$.
Then for every $m>0$ there exists a minimiser $\rho_m\in \mathcal{A}_{m,1}$ of $\mathcal I$, with compact support, which is unique up to translations and is characterised by the following Euler-Lagrange conditions. There exists a constant $\lambda\in \R$ such that (except for $x$ in a set of measure zero)
\begin{align}\label{EL1-d}
\left(W\ast \rho_m  \right) (x) + m\, \frac{|x|^2}2&= \lambda \qquad \text{if } 0<\rho_m(x)<1,\\ \label{EL2-d}
\smallskip
\left(W\ast \rho_m   \right) (x)+ m\, \frac{|x|^2}2&\le \lambda \qquad \text{if } \rho_m(x)=1,\\ \label{EL3-d}
\smallskip
\left(W\ast \rho_m   \right) (x)+ m\, \frac{|x|^2}2&\ge \lambda \qquad \text{if } \rho_m(x)=0.
\end{align}
\end{prop}

\begin{proof}  For the proof of existence we refer to \cite[Theorem~2.1]{CFT15}, which treats the special case of power-law potentials, but can be adapted to the present case. In particular, minimisers have compact support, due to the growth of the attraction term (see also \cite[Lemma~4.4]{BCT18} and \cite[Proposition~2.4]{CPT-Preprint}). On densities with compact support and 
barycentre at the origin the energy $\mathcal I$ reduces to
$$
{\mathcal I}(\rho)=I_m(\rho)= \int_{\R^d}\int_{\R^d} W(x-y) \rho(x)\rho(y)\,d x dy +m\int_{\R^d}|x|^2 \rho(x)\,dx.
$$
Uniqueness, up to translations, follows from the strict convexity of $I_m$ in the larger class $m\mathcal{P}(\R^d)$ of measures with compact support and finite interaction energy, which contains the class of densities in $\mathcal{A}_{m,1}$ that are relevant for the minimisation problem.

For the derivation of the Euler-Lagrange conditions we proceed as in \cite[Lemma~4.2]{BCT18}.
In a nutshell, given a minimiser $\rho_m\in \mathcal{A}_{m,1}$ of $\mathcal I$, for $m>0$, we construct variations of $\rho_m$ by perturbing it by means of a nonnegative function (smaller than $1$) in the set $\mathcal{S}_0=\{x: \rho_m(x)=0\}$, and by means of a nonpositive function (larger than $-1$) in the set $\mathcal{S}_1=\{x: \rho_m(x)=1\}$, in such a way that the resulting competitor has still mass $m$. This procedure leads to the conditions \eqref{EL1-d}--\eqref{EL3-d}, which are necessary for the minimality of $\rho_m$. Sufficiency follows from the strict convexity of the energy $I_m$ in the (convex) class $\mathcal{A}_{m,1}$.
\end{proof}

Finally, the minimisation of $\mathcal I$ on sets can be reduced to the relaxed problem on $\mathcal{A}_{m,1}$, by \cite[Theorem~4.5]{BCT18}, which we state below.

\begin{theorem}\label{th:equivalence}
Let $W$ be as in \eqref{potdef-dd} with $\kappa$ satisfying Assumption \ref{assumption-k}. %Assume that $\widehat{W}\geq 0$ on $\mathbb{S}^{d-1}$.
Then for every $m>0$ the energy $\mathcal I$ has a minimiser $\Omega\in \mathcal{A}_m$ if and only if the characteristic function $\chi_\Omega$ of $\Omega$ is a minimiser of $\mathcal I$ in $\mathcal{A}_{m,1}$.
\end{theorem}

\subsection{The Fourier transform} The Fourier transform definition we adopt is 
$$
\widehat{f}(\xi) =\frac{1}{(2\pi)^{d/2}}\int_{\R^d} f(x) e^{-i \xi \cdot x}\, dx, \quad \xi \in \R^d,
$$
for functions $f$ in the Schwartz space $\mathcal{S}$.
Correspondingly, the inverse Fourier transform is the following:
\begin{equation}\label{inv-F}
f(x)=\frac{1}{(2\pi)^{d/2}} \int_{\R^d} \widehat{f}(\xi) e^{i \xi \cdot x}\, d\xi, \quad x \in \R^d.
\end{equation}

\subsection{Ellipsoids} For any $a\in \R^d$ with $a\cdot e_i=a_i$, $D:=D(a)$ stands for the diagonal matrix such that $D_{ii}=a_i$. Given $a\in \R^d$ with $a_i>0$, we let 
\begin{equation}\label{eld}
E_0(a):= \left\{x\in \R^d:\ \sum_{i=1}^d\frac{x_i^2}{a_i^2} \le 1\right\}
\end{equation}
denote the compact set enclosed by the ellipsoid with semi-axes of length $a_i$ on the coordinate axis $e_i$. Note that 
\begin{equation*}%\label{E0-ball}
E_0(a) = D(a)\overline{B},
\end{equation*}
where $B$ denotes the unit ball $B_1(0)\subset\R^d$.
A general ellipsoid $E\subset \R^d$ centred at the origin can be then obtained by rotating $E_0(a)$ in \eqref{eld} with respect to the coordinate axes, namely as
\begin{equation}\label{ellrotd}
E= R E_0(a)=R D(a) \overline{B},
\end{equation}
for some rotation $R\in SO(d)$.

In the following we recall the Fourier transform of the (normalised) characteristic function of an ellipsoid. We refer to \cite{Mora-Muenster, MMRSV-3d} for the detailed computations in the cases $d=2$ and $d=3$ respectively, and we only summarise the results we need here. 

The Fourier transform of the characteristic function of $B$ in $\R^d$ is given by
\begin{equation}\label{Jalpha}
\widehat{\chi_{B}}(\xi)=\frac{J_{d/2}(|\xi|)}{|\xi|^{d/2}},
\end{equation}
where $J_{\alpha}$ denotes the Bessel function of the first kind of order $\alpha$.
We recall that $\frac{J_{d/2}(r)}{r^{d/2}}$ behaves as a constant at the origin, and as $r^{-(d+1)/2}$ at infinity.

Let $E$ be an ellipsoid of the form \eqref{ellrotd}. It is immediate to see that
\begin{equation}\label{chiE-chiBd}
\widehat{\chi_E}(\xi)=\frac{|E|}{|B|}\widehat{\chi_{B}}(D(a)R^T\xi).
\end{equation}

\subsection{The potential}\label{sec:potential-dd}

Let $W$ be as in \eqref{potdef-dd} with $\kappa$ satisfying Assumption \ref{assumption-k}. We recall that, for $\xi\in \R^d\setminus\{0\}$,
$$
\widehat{W}(\xi) = \frac{1}{|\xi|^2}\widehat{W}\left(\frac{\xi}{|\xi|}\right), 
$$
namely the Fourier transform of $W$ is $(-2)$-homogenous (see \cite{Mora-Muenster, MMRSV-3d}). We define the $0$-homogeneous function $\Psi: \R^d\setminus\{0\}\to \R$ given by $\Psi(y):=\widehat W(y)$ for $y\in \mathbb{S}^{d-1}$, so that 
\begin{equation}\label{FTW-dd}
\widehat{W}(\xi) = \frac{1}{|\xi|^2}\Psi(\xi), \quad \text{for } \xi \neq 0.
\end{equation}
Let $E$ be an ellipsoid of the form \eqref{ellrotd}. A crucial step in our approach is to resort to a Fourier representation of the potential $W\ast \chi_E$, which features in the Euler-Lagrange conditions \eqref{1}--\eqref{3}. Intuitively, one would like to use the inversion formula \eqref{inv-F} to write, for  $x \in \R^d$,
\begin{align*}
(W\ast\chi_E)(x) = \frac{1}{(2\pi)^{d/2}} \int_{\R^d} \widehat{W\ast \chi_E}(\xi) e^{i \xi \cdot x}\, d\xi 
 = 
 \int_{\R^d} \widehat{W}(\xi)\widehat{\chi_E}(\xi) \cos(\xi \cdot x)\, d\xi,
\end{align*}
and then use \eqref{Jalpha}--\eqref{FTW-dd} to make the expression above more explicit. While this is possible for $d=3$, due to a good integrability of $\widehat{W\ast \chi_E}$ (see \cite[Section~2.5]{MMRSV-3d}), in the two-dimensional case we need to proceed slightly differently, by applying the inversion formula to the gradient of the potential, rather than the potential itself, to gain integrability. Namely we compute for $x\in \R^2$ 
\begin{align*}
(\nabla W\ast \chi_E)(x)=\int_{\R^2}\widehat{\nabla W}(\xi)\widehat{\chi_E}(\xi)e^{ix\cdot\xi}\, d\xi=-\int_{\R^2}\xi \widehat{W}(\xi)\widehat{\chi_E}(\xi)\sin(x\cdot\xi)\, d\xi,
\end{align*}
see \cite{Mora-Muenster}. The expression of the potential on $E$ is given by
\begin{align}\label{pot-dd}
(W\ast\chi_E)(x) = 
\begin{cases}
\medskip
\displaystyle\ -\frac{|E|}{2|B|} \int_{\mathbb{S}^1}\alpha^2(x,y) \Psi(y)\,d\mathcal{H}^1(y) + c &\quad \text{for } d=2,\\
\displaystyle\frac{|E|}{2|B|}\sqrt{\frac{\pi}2}
\int_{\mathbb{S}^2}(1-\alpha^2(x,y))\,\frac{{\Psi}(y)}{{|D(a)R^Ty|}}\,d\mathcal{H}^2(y) &\quad \text{for } d=3,
\end{cases}
\end{align}
for $x\in E$, and for some constant $c\in \R$, where 
\begin{equation}\label{def-alpha}
\alpha(x,y):=\frac{x\cdot y}{|D(a)R^Ty|}.
\end{equation}
The gradient of the potential, instead, satisfies for any $x\in E^c$
\begin{align}\label{potgrad-dd}
&x\cdot\nabla(W\ast\chi_E)(x)\nonumber\\
= 
&\begin{cases}
\medskip
\displaystyle -\frac{|E|}{|B|}\int_{\mathbb{S}^1} 
\Psi(y) \left(
\alpha^2 \chi_{(-1,1)}(\alpha)+\frac{|\alpha|}{|\alpha|+\sqrt{\alpha^2-1}}\chi_{(-1,1)^c}(\alpha)
\right) d\mathcal{H}^1(y)&\quad \text{for } d=2,\\
\displaystyle-\frac{|E|}{|B|}\sqrt{\frac{\pi}2}
\int_{\mathbb{S}^2}\alpha^2\chi_{(-1,1)}(\alpha)\frac{{\Psi}(y)}{{|D(a)R^Ty|}}\,d\mathcal{H}^2(y) &\quad \text{for } d=3.
\end{cases}
\end{align}

\section{Proofs of our main results}\label{sec:main-res}
Proposition~\ref{ex+uniq-d} ensures that the energy $\mathcal I$ admits a minimiser in the class of densities $\mathcal{A}_{m,1}$ for every $m>0$. We recall that our aim is to consider the minimisation problem for $\mathcal I$ in the  smaller class $\mathcal{A}_m$ of sets of mass $m$. 

\subsection{The non-degenerate case: Proof of Theorem~\ref{Mainthm}}\label{sec:nd} In this section we assume that $\widehat{W}> 0$ on $\mathbb{S}^{d-1}$, hence we work in the non-degenerate case.
We now show that the minimiser provided by Proposition~\ref{ex+uniq-d} is the characteristic function of a set, and hence by Theorem~\ref{th:equivalence} it is a solution of the original problem in $\mathcal{A}_{m}$, only for masses $m$ above a given threshold.

\subsubsection{Proof of Theorem~\ref{Mainthm}: The critical and subcritical cases $m\leq m^\ast$.}\label{sect:sub}  We recall that under the non-degeneracy assumption $\widehat{W}> 0$ on $\mathbb{S}^{d-1}$ the 
unique minimiser, up to translations, of $\mathcal I$ on $\mathcal{P}(\R^d)$ is a (normalised) $d$-dimensional ellipsoid, both for $d=2$ (see \cite{CS2d,Mora-Muenster}) and $d=3$ (see \cite{CS3d,MMRSV-3d}).
We denote with $E^\ast:=R^*D(a^\ast)\overline B\subset \R^d$, for $a^\ast_i>0$, and $R^\ast\in SO(d)$, the ellipsoid such that 
$\mu^*=\chi_{E^*}/|E^*| \in \mathcal{P}(\R^d)$ is the minimiser of $\mathcal I$ on $\mathcal{P}(\R^d)$ with zero barycentre. We set $m^*:= |E^*|=|B| \prod_{i=1}^d a^*_i$.

For future reference we note that by strict convexity $\mu^*$ is also the unique measure satisfying the Euler-Lagrange conditions for the minimality of $\mathcal I$ on $\mathcal{P}(\R^d)$, namely
\begin{align}\label{old-EL1-d}
&(W\ast \mu^*)(x)+ \frac{|x|^2}2 = \lambda  \qquad \text{if } x\in \supp\mu^*,\\\label{old-EL2-d}
&(W\ast \mu^*)(x)+  \frac{|x|^2}2 \geq \lambda  \qquad \text{if } x\in \R^d\setminus \supp\mu^*,
\end{align}
for a constant $\lambda\in \R$.   

For every $m>0$ the measure $\mu_m:=\frac{m}{m^\ast}\chi_{E^\ast}$ is the unique minimiser of $\mathcal I$ on $m\mathcal{P}(\R^d)$, up to translations. This follows immediately from the quadratic behaviour of $\mathcal I(\mu)$ with respect to $\mu$.

Let now $m\leq m^*$. Since the function $\rho_{m}:=\mu_{m}=\frac{m}{m^\ast}\chi_{E^\ast}$ belongs to $\mathcal{A}_{m,1}$ in this case, we conclude that $\rho_{m}$ is the unique minimiser of $\mathcal I$ in the class $ \mathcal{A}_{m,1}$, up to translations. For $m=m^*$ this implies that $E^\ast\in \mathcal{A}_{m^\ast}$ is the unique minimiser of $I_{m^\ast}$  
in the class $ \mathcal{A}_{m^\ast}$, up to translations. For $m< m^\ast$ this proves non-existence of minimisers of $\mathcal I$ on $\mathcal A_m$ by Theorem~\ref{th:equivalence}, since 
$\rho_{m}$ is not the characteristic function of a~set.

Note that for $m>m^\ast$ the measure $\mu_m\in m \mathcal{P}(\R^d)\setminus  \mathcal{A}_{m,1}$, so one has to proceed differently to identify the minimiser of $\mathcal I$ in the class $\mathcal{A}_{m,1}$.

\subsubsection{Proof of Theorem~\ref{Mainthm}: The super-critical case $m> m^\ast$.}

We use the Euler-Lagrange conditions \eqref{EL1-d}--\eqref{EL3-d}, that we specialise in the case of $\rho=\chi_E$, with $E= R D(a)\overline B$, for $R\in SO(d)$ and $a_i>0$, satisfying $|E|=m$. More precisely, we show that there exist $R\in SO(d)$, and  $a\in \R^d$, with $a_i>0$ and $|B| \prod_{i=1}^d a_i=m$, such that 

\begin{align}\label{EL1E-d}
(W\ast \chi_E) (x) + m\frac{|x|^2}2&= \lambda \qquad \text{if } x\in \partial E,\\ \label{EL2E-d}
\smallskip
(W\ast \chi_E) (x) + m\frac{|x|^2}2&\le \lambda \qquad \text{if } x\in E^\circ,\\ 
\smallskip
(W\ast \chi_E) (x) + m\frac{|x|^2}2&\ge \lambda \qquad \text{if } x\in \R^d\setminus E, \label{EL3E-d}
\end{align}
for a constant $\lambda\in \R$.

%Note that, if $E$ satisfies \eqref{EL1E-d} and \eqref{EL3E-d}, then \eqref{EL2E-d} cannot hold with the equality sign. 
%Indeed, if this were the case, then the measure $\chi_E/|E|$ would be a solution of \eqref{old-EL1-d}--\eqref{old-EL2-d}. But then it would have to be $E=E^*$, by uniqueness, which is impossible since $|E|=m>m^*=|E^*|$.

We divide the proof into  a number of steps.\smallskip

\noindent
\textbf{Step~1: Conditions \eqref{EL1E-d}--\eqref{EL2E-d} imply \eqref{EL3E-d}.}
Assume that there exist $a_i>0$ with $|B| \prod_{i=1}^d a_i=m$, and $R\in SO(d)$, such that $E=R D(a)\overline B$ satisfies \eqref{EL1E-d}--\eqref{EL2E-d}, and let $x\in \R^d\setminus E$. 

We write $x=tx_0$, with $x_0\in\partial E$, and $t>1$. Since $x_0\in\partial E$, by \eqref{EL1E-d}--\eqref{EL2E-d}
we know that $x_0\cdot \nabla(W\ast\chi_E)(x_0)+m|x_0|^2\geq0$. We next treat the cases $d=2$ and $d=3$ separately.\smallskip

\textit{Step~1.1: The two-dimensional case.} By \eqref{pot-dd}, condition $x_0\cdot \nabla(W\ast\chi_E)(x_0)+m|x_0|^2\geq0$  becomes 
\begin{equation*}
-\frac{m}{|B|}
\int_{\mathbb{S}^1}\alpha^2(x_0,y) \Psi(y)\,d\mathcal{H}^1(y)+m|x_0|^2\geq0,
\end{equation*}
where $\alpha(x,y)$ is defined in \eqref{def-alpha}. This implies in particular that 
$$
|x_0|^2 \geq \frac1{|B|}\int_{\mathbb{S}^1}\alpha^2(x_0,y) \Psi(y)\,d\mathcal{H}^1(y),
$$%
and hence 
\begin{equation}\label{eq10}
|x|^2 =t^2|x_0|^2 \geq \frac{t^2}{|B|}\int_{\mathbb{S}^1}\alpha^2(x_0,y) \Psi(y)\,d\mathcal{H}^1(y)
= \frac{1}{|B|}\int_{\mathbb{S}^1}\alpha^2(x,y) \Psi(y)\,d\mathcal{H}^1(y),
\end{equation}
since $t^2 \alpha^2(x_0,y) = \alpha^2(tx_0,y)= \alpha^2(x,y)$. Then, by \eqref{potgrad-dd} and \eqref{eq10}
\begin{align*}
x\cdot \left(\nabla W\ast \chi_E   \right) (x)+ m\, |x|^2
 \geq &  -\frac{m}{|B|}\int_{\mathbb{S}^1} 
\Psi(y) \left(
\alpha^2 \chi_{(-1,1)}+\frac{|\alpha|}{|\alpha|+\sqrt{\alpha^2-1}}\chi_{(-1,1)^c}
\right) d\mathcal{H}^1(y)\\
&\quad + \frac{m}{|B|}\int_{\mathbb{S}^1}\alpha^2\Psi(y)\,d\mathcal{H}^1(y)\\
&= \frac{m}{|B|}\int_{\mathbb{S}^1} \Psi(y) \chi_{(-1,1)^c}(\alpha) \left(\alpha^2- \frac{|\alpha|}{|\alpha|+\sqrt{\alpha^2-1}}\right) d\mathcal{H}^1(y)\geq 0.
\end{align*}
This implies \eqref{EL3E-d} and concludes the proof.\smallskip

\textit{Step~1.2: The three-dimensional case.} By \eqref{pot-dd}, condition $x_0\cdot \nabla(W\ast\chi_E)(x_0)+m|x_0|^2\geq0$  becomes 
$$
-\frac{m}{|B|}\sqrt{\frac{\pi}2}
\int_{\mathbb{S}^2}\alpha^2(x_0,y)\frac{{\Psi}(y)}{{|D(a)R^Ty|}}\,d\mathcal{H}^2(y)+m|x_0|^2\geq 0.
$$
This implies in particular that 
\begin{align*}%\label{eq10}
|x|^2 =t^2|x_0|^2 &\geq \frac{t^2}{|B|}\sqrt{\frac{\pi}2}
\int_{\mathbb{S}^2}\alpha^2(x_0,y)\frac{{\Psi}(y)}{{|D(a)R^Ty|}}\,d\mathcal{H}^2(y)\\
&= \frac{1}{|B|}\sqrt{\frac{\pi}2}\int_{\mathbb{S}^2}\alpha^2(x,y)\frac{{\Psi}(y)}{{|D(a)R^Ty|}}\,d\mathcal{H}^2(y),
\end{align*}
since $t^2 \alpha^2(x_0,y) = \alpha^2(tx_0,y)= \alpha^2(x,y)$. 

Then, by \eqref{potgrad-dd} and by the inequality above
\begin{align*}
x\cdot \left(\nabla W\ast \chi_E   \right) (x)+ m\, |x|^2
 \geq & -\frac{m}{|B|}\sqrt{\frac{\pi}2}
\int_{\mathbb{S}^2}\alpha^2(x,y)\chi_{(-1,1)}(\alpha)\frac{{\Psi}(y)}{{|D(a)R^Ty|}}\,d\mathcal{H}^2(y)\\
&\quad +  \frac{m}{|B|}\sqrt{\frac{\pi}2}\int_{\mathbb{S}^2}\alpha^2(x,y)\frac{{\Psi}(y)}{{|D(a)R^Ty|}}\,d\mathcal{H}^2(y)\\
&= \frac{m}{|B|}\sqrt{\frac{\pi}2}
\int_{\mathbb{S}^2}\alpha^2(x,y)\chi_{(-1,1)^c}(\alpha)\frac{{\Psi}(y)}{{|D(a)R^Ty|}}\,d\mathcal{H}^2(y)\geq 0.
\end{align*}
This implies \eqref{EL3E-d} and concludes the proof of this step.

Note that this step is valid also in the degenerate case $\widehat W\geq0$ on $\mathbb S^{d-1}$.\smallskip

\noindent
\textbf{Step~2: Rewriting conditions \eqref{EL1E-d}--\eqref{EL2E-d}.} Let $\mathbb{M}_+$ denote the space of positive definite symmetric $d\times d$ matrices. We set $M:= RD^2R^T$, where $D^2=D(a^2)$. Note that $M\in \mathbb{M}_+$, and $\det M=\prod_{i=1}^d a_i^2=\frac{m^2}{|B|^2}$, since $m=|E|=|B| \prod_{i=1}^d a_i$.

We now rewrite the potential on $E$ in \eqref{pot-dd} in terms of $M$. First note that  
\begin{align}\label{Da-M}
|D(a)R^Ty| &= (D(a)R^Ty\cdot D(a)R^Ty)^{1/2}=((D(a)R^T)^T D(a)R^Ty\cdot y)^{1/2} \nonumber\\
&=(RD(a^2)R^Ty\cdot y)^{1/2} = (My\cdot y)^{1/2}.
\end{align}
Hence, from \eqref{pot-dd}, the potential in terms of $M$ is 
\begin{equation}\label{WEd}
(W\ast \chi_E)(x)=-\frac{m}{2|B|}\gamma_d \int_{\mathbb{S}^{d-1}} \frac{(x\cdot y)^2\Psi(y)}{(My\cdot y)^{d/2}}\,d\mathcal{H}^{d-1}(y)+c,
\end{equation}
for $x\in E$, for some constant $c\in \R$, where $\gamma_d$ is a constant defined as 
\begin{equation}\label{def-gamma}
\gamma_d:=
\begin{cases}
\medskip
1 &\quad \text{for } d=2,\\
\displaystyle\sqrt{\frac{\pi}2} &\quad \text{for } d=3.
\end{cases}
\end{equation}
More concisely, let $Q\in \R^{d\times d}$  be defined as the matrix with components $Q_{ij}:=q_{ij}$,
where 
\begin{align}\label{AAc-d}
q_{ij}:=-\frac{m}{2|B|}\gamma_d \int_{\mathbb{S}^{d-1}} \frac{\Psi(y)y_iy_j}{(My\cdot y)^{d/2}}\,d\mathcal{H}^{d-1}(y).
\end{align}
Note that $Q$ depends on the potential $W$ and on the matrix $M$, but we will not indicate this dependence explicitly to avoid overburdening the notation. 
Then we can write \eqref{WEd}~as 
\begin{equation}\label{quad:pot}
(W\ast \chi_E)(x) = Qx\cdot x + c
\end{equation}
for $x\in E$, and for some constant $c\in \R$. Hence we can rewrite \eqref{EL1E-d}--\eqref{EL2E-d} in terms of $Q$ as 
\begin{align}\label{EL1E-A-d}
\left(Q+\frac m2 I\right)x\cdot x&= \lambda \qquad \text{if } x\in \partial E,\\ \label{EL2E-A-d}
\medskip
\left(Q+\frac m2 I\right)x\cdot x&\le \lambda \qquad \text{if } x\in E^\circ,
\end{align}
for some constant $\lambda\in \R$ (not renamed, corresponding to $\lambda-c$, where $\lambda$ is the constant in \eqref{EL1E-d}--\eqref{EL2E-d}).

Note that, by \eqref{ellrotd}, $x\in  E$ if and only if $x=RD(a)z= M^{1/2}Rz$, with $z\in \overline B$, where $M^{1/2}:=R D(a)R^T$. In fact, by setting $\xi:=R z$, we can write $x=M^{1/2}\xi$, with $\xi\in \overline B$. Then, for $x\in E$, we have 
\begin{align}\label{A-sqM3}
\left(Q+\frac m2 I\right)x\cdot x = M^{1/2}\left(Q+\frac m2 I\right)M^{1/2}\xi\cdot \xi,
\end{align}
with $\xi\in \overline B$. By \eqref{A-sqM3}, conditions \eqref{EL1E-A-d}--\eqref{EL2E-A-d} are equivalent to 
\begin{align}\label{EL1E-M-d}
M^{1/2}\left(Q+\frac m2 I\right)M^{1/2}\xi\cdot \xi&= \lambda \qquad \text{if } \xi\in \partial B,\\ 
\label{EL2E-M-d}
\medskip
M^{1/2}\left(Q+\frac m2 I\right)M^{1/2}\xi\cdot \xi&\le \lambda \qquad \text{if } \xi\in B.
\end{align}
Since condition \eqref{EL1E-M-d} is equivalent to 
\begin{equation}\label{AM-inv-d}
M^{1/2}\left(Q+\frac m2 I\right)M^{1/2} = \lambda I,
\end{equation}
condition \eqref{EL2E-M-d} rewrites as 
\begin{equation*}%\label{AM-inv-1}
\lambda\xi\cdot \xi \leq \lambda \quad \text{if } \xi\in B,
\end{equation*}
that is, $\lambda\geq 0$.

In conclusion, finding an ellipsoid $E$ satisfying \eqref{EL1E-d}--\eqref{EL2E-d} is equivalent to finding $M\in \mathbb{M}_+$, with $\det M=\frac{m^2}{|B|^2}$, and satisfying \eqref{AM-inv-d} with $\lambda\geq 0$, namely satisfying 
\begin{equation}\label{rewritingEL-d}
Q+\frac m2 I = \lambda M^{-1}, \quad \lambda\geq 0.
\end{equation}\smallskip

\noindent
\textbf{Step~3: Rewriting conditions \eqref{rewritingEL-d} as a constrained minimisation problem.} By writing explicitly
$$
M^{-1}=\frac{1}{\det M} \text{adj}(M),
$$
where $\text{adj}(M)$ denotes the adjugate of $M$, we obtain the following scalar conditions, equivalent to \eqref{rewritingEL-d}:
\begin{align}\label{components1-d}
&q_{ij}+\frac m2\delta_{ij} = \frac{\lambda}{\det M} (\text{adj}(M))_{ij}, \quad \text{for } i,j=1,\dots, d,  \\\label{components3-d}
&\lambda\geq 0.
\end{align}
Finally, by setting $\tilde \lambda:=\frac{2\lambda}{m \det M}$, and by using \eqref{AAc-d}, conditions \eqref{components1-d}--\eqref{components3-d} become 
\begin{align}\label{scalar1-d}
&-\frac{1}{|B|}\gamma_d \int_{\mathbb{S}^{d-1}} \frac{\Psi(y)y_jy_j}{(My\cdot y)^{d/2}}\,d\mathcal{H}^{d-1}(y)+\delta_{ij}=\tilde \lambda (\text{adj}(M))_{ij}, \quad i,j=1,\dots,d,
\\\label{scalar3-d}
&\quad \tilde\lambda \geq0,
\end{align}
where $M\in \mathbb{M}_+$ satisfies $\det M=\frac{m^2}{|B|^2}$.

We now show that \eqref{scalar1-d}--\eqref{scalar3-d} can be written as the stationarity condition for an auxiliary scalar function defined on matrices $M\in \mathbb{M}_+$, under the determinant constraint $\det M=\frac{m^2}{|B|^2}$. To this aim we introduce the continuous function $f:\mathbb{M}_+\to \R$ defined as 
\begin{equation}\label{def-fg}
f(M):= g_d(M)+ \mathrm{tr}(M), 
\end{equation}
where
\begin{equation}\label{def-f-d}
g_d(M):=
\begin{cases}
\medskip
\displaystyle -\frac{1}{|B|}\int_{\mathbb{S}^1}\Psi(y)\log(My\cdot y)\,d\mathcal{H}^1(y) &\quad \text{for } d=2,\\
\displaystyle\frac{2\gamma_3}{|B|}\int_{\mathbb{S}^2}\frac{\Psi(y)}{\sqrt{My\cdot y}}\,d\mathcal{H}^2(y)  &\quad \text{for } d=3.
\end{cases}
\end{equation}
One can show that $f$ is nothing but the nonlocal energy $\mathcal I$ evaluated at normalised characteristic functions of ellipsoids, see Lemma~\ref{lemma}.
Then \eqref{scalar1-d}--\eqref{scalar3-d} can be rewritten as
\begin{align}\label{new-claim-d}
\nabla_M f(M)=\tilde \lambda \nabla_M\det(\cdot)(M), \quad 
\tilde\lambda \geq0,
\end{align}
where we denoted
$$\nabla_M=
\begin{cases}
\medskip
\displaystyle\left(\frac{\partial}{\partial M_{11}}, \frac{\partial}{\partial M_{22}},\frac{\partial}{\partial M_{12}}\right)&\quad \text{for } d=2,\\\displaystyle\left(\frac{\partial}{\partial M_{11}},\frac{\partial}{\partial M_{22}},\frac{\partial}{\partial M_{33}},\frac{\partial}{\partial M_{12}},\frac{\partial}{\partial M_{13}},\frac{\partial}{\partial M_{23}}\right) &\quad \text{for } d=3,\\
\end{cases}
$$
since due to the symmetry of $M$
$$
\frac{\partial}{\partial M_{ij}} \det(\cdot)(M)= (\text{adj}(M))_{ji}=(\text{adj}(M))_{ij},\quad i,j=1,\dots,d.
$$

We have then reduced the claim of finding an ellipsoid $E$ satisfying \eqref{EL1E-d}--\eqref{EL2E-d} to finding $M\in \mathbb{M}_+$ with $\det M=\frac{m^2}{|B|^2}$, satisfying \eqref{new-claim-d}.\smallskip

\textit{Step~3.1: Auxiliary constrained minimisation problem for $f$.} We now show that the function $f$ in \eqref{def-fg} admits a minimiser in the (larger) set 
\begin{equation}\label{unilater}
\mathcal{M}:= \left\{M\in \mathbb{M}_+: \ \det M\geq \frac{m^2}{|B|^2}\right\},
\end{equation}
where we replaced the equality constraint for the determinant of $M$ with a unilateral constraint.

To this aim, let $(M_n)_n\subset \mathcal{M}$ be a minimising sequence for $f$, where we write $M_n=: R_nD(a_n^2)R_n^T$, with $R_n\in SO(d)$ and $a_n\in \R^d$ with $a_n\cdot e_i>0$. Since 
$$
\inf_{\mathcal{M}} f \leq f \bigg(\bigg(\frac{m}{|B|}\bigg)^{2/d} I\bigg)<+\infty,
$$
we have that $f(M_n)\leq c<\infty$, for some constant $c\in \R$. We now show that 
\begin{align}\label{en:bound}
c\geq f(M_n) \geq C |M_n|,
\end{align}
where $C>0$, and $|\cdot|$ denotes the Frobenius norm of the matrix. Then 
$$
|M_n|^2=\text{tr}(M_n^TM_n)=\text{tr}(D(a_n^4))=\sum_{i=1}^d(a_n)_i^4,
$$
and since 
$$
\mathrm{tr}(M_n)=\sum_{i=1}^d(a_n)_i^2,
$$
we immediately deduce that  
\begin{equation}\label{norm-trace}
|M_n|\leq\text{tr}(M_n).
\end{equation}
In the case $d=3$, since $g_3(M_n)\geq 0$, the bound \eqref{norm-trace} gives immediately \eqref{en:bound}. 

If $d=2$, we have that  
\begin{align}\label{estimateg_2}
g_2(M_n) &= -\frac{1}{|B|}\int_{\mathbb{S}^1}\Psi(y)\log\left(\frac{M_n}{|M_n|}y\cdot y\right)\,d\mathcal{H}^1(y)
-\log(|M_n|)\left(\frac{1}{|B|}
\int_{\mathbb{S}^1} \Psi(y)d\mathcal{H}^1(y)
\right)
\nonumber
\\
&\geq  -\log(|M_n|)\left(\frac{1}{|B|}
\int_{\mathbb{S}^1} \Psi(y)d\mathcal{H}^1(y)
\right),
\end{align}
where for the last inequality we have used that for every $y\in \mathbb{S}^1$
$$
0<\frac{M_n}{|M_n|}y\cdot y \leq 1,
$$
and hence its logarithm is negative. We also recall that 
\begin{equation}\label{average-Psi}
\int_{\mathbb{S}^1} \Psi(y)d\mathcal{H}^1(y) = 2\pi
\end{equation}
(see, e.g., \cite[eq.\ (3.6)]{Mora-Muenster}),
and so, from \eqref{estimateg_2} and \eqref{norm-trace}, in the case $d=2$ we obtain the estimate 
\begin{align*}
c\geq f(M_n) &= g_2(M_n)+\text{tr}(M_n) \geq -2\log(|M_n|)
+ |M_n| \geq \frac{|M_n|}4.
\end{align*}
Hence, also in this case we obtain the bound \eqref{en:bound}.

Both for $d=2$ and $d=3$, up to a subsequence, we have that $M_n\to M_0$ as $n\to +\infty$, where $M_0\in \R^{d\times d}$ is symmetric, but a priori only semi-positive definite; however, since it satisfies the unilateral constraint $\det M\geq \frac{m^2}{|B|^2}$, we have that in fact $M_0\in \mathcal{M}$. Since the function $f$ in \eqref{def-fg} is continuous, we have proved that $f$ admits a minimiser in~$\mathcal M$.\smallskip

\textit{Step~3.2: Conclusion.} By the Karush-Kuhn-Tucker conditions for the auxiliary constrained minimisation problem in Step~3.1, we have that any minimiser of $f$ in the set $\mathcal{M}$ satisfies 
\begin{align}\label{new-claim-KKT-1d}
&\nabla_M f(M)=\tilde \lambda \nabla_M\det(\cdot)(M), \\
\label{new-claim-KKT-2d}
&\tilde\lambda \geq0, \quad 
\det M\geq \frac{m^2}{|B|^2}, \quad 
\tilde \lambda\left(\det M-\frac{m^2}{|B|^2}\right)=0.
\end{align}
To conclude the proof of the theorem it suffices to show that it must be $\det M=\frac{m^2}{|B|^2}$.

Assume, for contradiction, that $\det M>\frac{m^2}{|B|^2}$. Then by \eqref{new-claim-KKT-2d} we deduce that $\tilde \lambda=0$, and hence $\nabla_M f(M)=0$ by \eqref{new-claim-KKT-1d}. 
In other words, $M$ is a critical point of $f$. This implies that the corresponding ellipsoid $E$ satisfies \eqref{EL1E-d} and \eqref{EL2E-d} with the equality sign, and by Step~1 it satisfies
\eqref{EL3E-d}, too. Therefore, the measure $\chi_E/|E|$ is a solution of \eqref{old-EL1-d}--\eqref{old-EL2-d}. Hence, by uniqueness $E=E^*$ and
$$
\det M = \prod_{i=1}^da_i^2 = \frac{(m^*)^2}{|B|^2}.
$$
But this is in contradiction with the assumption $\det M>\frac{m^2}{|B|^2}$, since we are dealing with the case $m>m^*$.\smallskip

\textbf{Step~4: The limiting case $m\to + \infty$.}  Let $M= RD(a^2)R^T\in \mathbb{M}_+$ be the matrix corresponding to the ellipsoid $E$ that solves \eqref{EL1E-d}--\eqref{EL3E-d}, where $R\in SO(d)$, and  $a\in \R^d$, with $a_i>0$ and $|B| \prod_{i=1}^d a_i=m$. By \eqref{AM-inv-d}, since $M^{1/2}=R D(a)R^T$, we deduce that 
\begin{equation}\label{interm}
RD(a)R^T\left(Q+\frac m2 I\right)RD(a)R^T = \lambda I,
\end{equation}
By a simple change of variables we can write $Q=\frac m2(RPR^T)$, where $P=(p_{ij})$ is defined as 
\begin{align}\label{Pc-d}
p_{ij}:=-\frac{1}{|B|}\gamma_d \int_{\mathbb{S}^{d-1}} \frac{\Psi(Ry)y_iy_j}{(D(a^2)y\cdot y)^{d/2}}\,d\mathcal{H}^{d-1}(y).
\end{align}
Rewriting \eqref{interm} in terms of $P$ gives 
\begin{equation}\label{PM-inv-d}
\frac m2(D(a)PD(a)+ D(a^2)) = \lambda I.
\end{equation}
For the diagonal components of \eqref{PM-inv-d} we find 
$$
-\frac{1}{|B|}\gamma_d \int_{\mathbb{S}^{d-1}} \frac{\Psi(Ry)a_i^2 y_i^2}{(D(a^2)y\cdot y)^{d/2}}\,d\mathcal{H}^{d-1}(y)+a_i^2 =\frac{2\lambda}m.
$$
By subtracting the component $j$ from the component $i$ in the expression above and by rearranging we obtain 
\begin{align*}
a_i^2-a_j^2\leq\frac{\gamma_d}{|B|} \int_{\mathbb{S}^{d-1}} \frac{\Psi(Ry)}{(D(a^2)y\cdot y)^{\frac d2-1}}\,d\mathcal{H}^{d-1}(y) - \frac{\gamma_d}{|B|} \int_{\mathbb{S}^{d-1}} \frac{\Psi(Ry) \sum_{k\neq i}a_k^2y_k^2}{(D(a^2)y\cdot y)^{d/2}}\,d\mathcal{H}^{d-1}(y)
\end{align*}
and 
\begin{align*}
a_i^2-a_j^2\geq-\frac{\gamma_d}{|B|} \int_{\mathbb{S}^{d-1}} \frac{\Psi(Ry)}{(D(a^2)y\cdot y)^{\frac d2-1}}\,d\mathcal{H}^{d-1}(y) + \frac{\gamma_d}{|B|} \int_{\mathbb{S}^{d-1}} \frac{\Psi(Ry) \sum_{k\neq j}a_k^2y_k^2}{(D(a^2)y\cdot y)^{d/2}}\,d\mathcal{H}^{d-1}(y),
\end{align*}
hence
$$
|a_i^2-a_j^2| \leq \frac{\gamma_d}{|B|} \int_{\mathbb{S}^{d-1}} \frac{\Psi(Ry)}{(D(a^2)y\cdot y)^{\frac d2-1}}\,d\mathcal{H}^{d-1}(y).
$$
Let $t\in \R^d$ denote the vector with components $t_i:= \left(\frac{|B|}m\right)^{\frac1d}a_i$. In terms of $t$, the estimate above becomes 
\begin{equation}\label{estimate:rescaled-axes}
|t_i^2-t_j^2| \leq \frac{\gamma_d}{m} \int_{\mathbb{S}^{d-1}} \frac{\Psi(Ry)}{(D(t^2)y\cdot y)^{\frac d2-1}}\,d\mathcal{H}^{d-1}(y).
\end{equation}
We now show that the integral in \eqref{estimate:rescaled-axes} is bounded uniformly in $m$. Since $M$ is the matrix corresponding to the ellipsoid $E$ that solves \eqref{EL1E-d}--\eqref{EL3E-d}, we have by Step~3 that 
$$
c|M|\leq f(M)= \min_{\mathcal{M}} f \leq f \bigg(\bigg(\frac{m}{|B|}\bigg)^{2/d} I\bigg) = g_d\bigg(\bigg(\frac{m}{|B|}\bigg)^{2/d} I\bigg)+d\bigg(\frac{m}{|B|}\bigg)^{2/d},
$$
where we have used \eqref{def-fg} and \eqref{en:bound}, and where the constant $c>0$ is independent of $m$. Since
$\mathrm{tr}(M)\leq d |M|$, the previous estimate leads to 
\begin{align}\label{almost-trace}
\sum_{i=1}^d a_i^2=\mathrm{tr}(M) \leq C\bigg(g_d\bigg(\bigg(\frac{m}{|B|}\bigg)^{2/d} I\bigg)+d\bigg(\frac{m}{|B|}\bigg)^{2/d}\bigg),
\end{align}
for a constant $C>0$ independent of $m$. Finally, by \eqref{def-f-d} and \eqref{average-Psi} we have  
$$
g_d\bigg(\bigg(\frac{m}{|B|}\bigg)^{2/d} I\bigg)=
\begin{cases}
\medskip
\displaystyle-2\log \bigg(\frac{m}{|B|}\bigg) &\quad \text{for } d=2,\\
\displaystyle \bigg(\frac{|B|}{m}\bigg)^{1/3} \frac{2\gamma_3}{|B|}\int_{\mathbb{S}^2} \Psi(y)\,d\mathcal{H}^2(y)  &\quad \text{for } d=3,
\end{cases}
$$
hence this term is negligible in the right-hand side of \eqref{almost-trace}, for $m$ large. Therefore, for $m\gg1$, from \eqref{almost-trace} we deduce the bound 
\begin{align}\label{est-trace}
\sum_{i=1}^d a_i^2 \leq C \bigg(\frac{m}{|B|}\bigg)^{2/d},
\end{align}
which in terms of the rescaled semi-axes $t$ becomes
\begin{align}\label{est-trace-t}
\sum_{i=1}^d t_i^2 \leq C.
\end{align}
In particular, this means that the sequences $(t_i)$ are uniformly bounded in $m$ for $i=1,\dots, d$. Moreover, since $\det M = \prod_{i=1}^da_i^2 = \frac{m^2}{|B|^2}$, we have that 
$$
\prod_{i=1}^dt_i^2 =\frac{|B|^2}{m^2} \prod_{i=1}^da_i^2 = 1,
$$
and, consequently, $t_i$ are also bounded away from zero, uniformly in $m$. Hence, up to a subsequence, $t_i\to \bar t_i$ as $m\to +\infty$, with $0<\bar t_i<+\infty$. This implies in particular that the integral in \eqref{estimate:rescaled-axes} is bounded uniformly in $m$, and finite in the limit $m\to +\infty$, by the dominated convergence theorem. By letting $m\to +\infty$ in \eqref{estimate:rescaled-axes} we finally obtain that $\bar t_i=\bar t_j=1$ for every $i,j=1,\dots,d$.

This shows that in the limit, as the prescribed mass $m$ tends to $+\infty$, the minimising ellipsoid converges to a ball. 

\subsection{The degenerate case: Proof of Theorem~\ref{Mainthm2}}  In this section we assume that $\widehat{W}\geq 0$ on $\mathbb{S}^{d-1}$, hence we work in the degenerate case. 
We distinguish two cases, corresponding to the two possible minimisers for $\mathcal I$ on measures.

If $\mathcal I$ is minimised on $\mathcal{P}(\R^d)$ by the normalised characteristic function of a $d$-dimen\-sion\-al ellipsoid $E^\ast$, then one can easily see that the whole procedure detailed in Section~\ref{sec:nd} follows without change. In particular, there exists a threshold $m^*$ (the mass of $E^\ast$) such that the minimiser of $\mathcal I$ is an ellipsoid for $m\geq m^*$, and is a density in $\mathcal{A}_{m,1}\setminus \mathcal{A}_m$ for $m<m^*$. 

If, instead, $\mathcal I$ is minimised on $\mathcal{P}(\R^d)$ by a measure with a $(d-1)$-dimensional support, then one can repeat the whole argument of Section~\ref{sec:nd} up to Step~3.1. To conclude one has to show that for any prescribed mass $m>0$ the minimiser $M$ of $f$ in the set $\mathcal{M}$ satisfies $\det M=\frac{m^2}{|B|^2}$.
Assume, for contradiction, that $\det M>\frac{m^2}{|B|^2}$. Then by \eqref{new-claim-KKT-2d} we deduce that $M$ is a critical point of $f$. This implies that the corresponding ellipsoid $E$ satisfies \eqref{EL1E-d} and \eqref{EL2E-d} with the equality sign, and by Step~1 it satisfies
\eqref{EL3E-d}, too. Therefore, the measure $\chi_E/|E|$ is a solution of \eqref{old-EL1-d}--\eqref{old-EL2-d}, hence it is a minimiser of $\mathcal I$ on $\mathcal{P}(\R^d)$, contradicting the assumption that minimisers are $(d-1)$-dimensional. Finally, the proof of Step~4 can be repeated verbatim.

\section{Further results}\label{sect:further}

In this section we prove some monotonicity properties of the optimal ellipsoids, and we investigate when the optimal ellipsoid is a ball. We also provide a bound for the critical mass and clarify the meaning of the auxiliary function $f$ which is instrumental in the proof of Theorems~\ref{Mainthm} and~\ref{Mainthm2}.

In the previous section we proved that optimal ellipsoids become more and more rounded in the limit $m\to +\infty$. The next remark shows that this convergence is in some sense monotone. Although it is not necessarily true that each of the semi-axes of the optimal ellipsoids increases as $m$ increases, they do so `collectively'. More precisely, we show that the sum of the squares of the semi-axes, suitably normalised, is increasing in $m$.

\begin{remark}[Monotonicity of the minimiser with respect to $m$]\label{rem:mon}
Let $0<m_1<m_2$. Assume that $\mathcal I$ is minimised by $E_1$ on $\mathcal A_{m_1}$ and by $E_2$ on $\mathcal A_{m_2}$.
Let $M_1$ and $M_2$ be the positive definite matrices associated to $E_1$ and $E_2$, respectively, as in Step~2 of the proof of Theorem~\ref{Mainthm}.
Then 
\begin{equation}\label{monotone}
\frac{\mathrm{tr}(M_1)}{(m_1)^{2/d}}\geq \frac{\mathrm{tr}(M_2)}{(m_2)^{2/d}}. 
\end{equation}
Indeed, from the proof of Theorems~\ref{Mainthm} and~\ref{Mainthm2} we have that $M_i$ minimises the function $f$ in \eqref{def-fg} on the set 
$$
\left\{M\in \mathbb{M}_+: \ \det M\geq \frac{m_i^2}{|B|^2}\right\}.
$$
%and satisfies $\det M_i= m_i^2/|B|^2$
Therefore, setting $\lambda_i:=(|B|/m_i)^{2/d}$, we obtain
$$
f(M_1)\leq f\left(\frac{\lambda_2}{\lambda_1}M_2\right), \qquad
f(M_2)\leq f\left(\frac{\lambda_1}{\lambda_2}M_1\right).
$$
Using the scaling properties
$$
f(\lambda M)= 
\begin{cases} 
\medskip
f(M)+ (\lambda-1)\mathrm{tr}(M)-2\log\lambda \qquad &\text{ for } d=2,\\
\frac1{\sqrt{\lambda}}f(M)+ \frac1{\sqrt{\lambda}}(\lambda^{3/2}-1)\mathrm{tr}(M) \qquad &\text{ for } d=3,
\end{cases}
$$
for $\lambda>0$, we deduce that
$$
\left(\frac{\mathrm{tr}(M_1)}{(m_1)^{2/d}}-\frac{\mathrm{tr}(M_2)}{(m_2)^{2/d}}\right)(m_2-m_1)\geq0, 
$$
which immediately gives \eqref{monotone}.
\end{remark}

We now investigate the minimality of balls for the energy $\mathcal I$. Note that Theorems~\ref{Mainthm} and~\ref{Mainthm2} guarantee that balls are `asymptotically' optimal, in the limit $m\to +\infty$, but they say nothing about the case of finite $m$. 

In what follows, we denote with $B_{R_m}$ the ball centred at the origin with radius $R_m=(m/|B|)^{1/d}$, so that $|B_{R_m}|=m$.

\begin{prop}[Minimality of the ball for finite mass $m$]\label{next:prop}
Let $W$ be as in \eqref{potdef-dd} with $\kappa$ satisfying Assumption \ref{assumption-k}, and such that $\widehat{W}\geq 0$ on $\mathbb{S}^{d-1}$. Let $\Psi: \mathbb{S}^{d-1}\to \R$ be defined as in \eqref{FTW-dd}, and assume that 
\begin{equation}\label{AAAp}
\int_{\mathbb{S}^{d-1}} \Psi(y)y_iy_j\,d\mathcal{H}^{d-1}(y) = \frac{\delta_{ij}}d \int_{\mathbb{S}^{d-1}} \Psi(y)\,d\mathcal{H}^{d-1}(y)
\end{equation}
for every $i,j=1,\dots,d$.
Let
\begin{equation}\label{mast-b}
m^\circ:=\frac{\gamma_d}d \int_{\mathbb{S}^{d-1}} \Psi(y)\,d\mathcal{H}^{d-1}(y).
\end{equation}
Then, up to translations, $\mu^\circ=\frac{1}{m^\circ}\chi_{B_{R_{m^\circ}}}\in \mathcal{P}(\R^d)$ is the unique minimiser of $\mathcal I$ on $\mathcal{P}(\R^d)$, and $B_{R_m}$ is the unique minimiser of  $\mathcal I$ on $\mathcal A_m$ for $m\geq m^\circ$.

If, instead, \eqref{AAAp} is not satisfied, then balls are never minimisers of  $\mathcal I$ on $\mathcal A_m$ for any mass $m>0$.
\end{prop}

\begin{proof}
Let $m>0$ be fixed, and assume that $\widehat{W}\geq 0$ on $\mathbb{S}^{d-1}$. We recall that, by the proof of Theorems~\ref{Mainthm} and~\ref{Mainthm2}, an ellipsoid $E\in \mathcal A_m$ centred at $0$ minimises $\mathcal I$ on $\mathcal A_m$ if and only if its associated matrix $M\in \mathbb{M}_+$ (constructed as in Step~2 of the proof of Theorem~\ref{Mainthm}) satisfies \eqref{rewritingEL-d},
where $Q\in \R^{d\times d}$ is the matrix with components $Q_{ij}=q_{ij}$ defined in \eqref{AAc-d}.
%We also recall that \eqref{rewritingEL-d2} is equivalent to \eqref{EL1E-d}--\eqref{EL2E-d}.

We now specify \eqref{rewritingEL-d} in the special case of $E=B_{R_m}$, namely when the ellipsoid is a ball of mass $m$. First of all, the matrix $M$ associated to $B_{R_m}$ is $M=(m/|B|)^{2/d}I$. Correspondingly, \eqref{AAc-d} becomes
\begin{equation}\label{qij-B}
q_{ij}= -\frac{\gamma_d}{2} \int_{\mathbb{S}^{d-1}} \Psi(y) y_i y_j\,d\mathcal{H}^{d-1}(y),
\end{equation}
which shows that in this case $Q$ does not depend on $m$, or equivalently, it does not depend on the radius of the ball. Moreover, since $M^{-1}=(|B|/m)^{2/d}I$, the equality in \eqref{rewritingEL-d} simplifies to
\begin{equation}\label{AAAb}
Q= \bigg(\lambda \left(\frac{|B|}{m}\right)^{2/d}-\frac m2 \bigg)I.
\end{equation}
Putting together \eqref{qij-B} and \eqref{AAAb} we can write $Q=qI$, with 
\begin{equation}\label{latter}
q= q_{ii}= -\frac{\gamma_d}{2d} \int_{\mathbb{S}^{d-1}} \Psi(y)\,d\mathcal{H}^{d-1}(y), \quad q_{ij}=0 \,\, \text{if } i\neq j.
\end{equation}
Condition \eqref{latter} is in fact equivalent to the equality in \eqref{rewritingEL-d}, and is a condition on the potential. In other words, for any $m>0$, the ball $B_{R_m}$ satisfies the equality condition in \eqref{rewritingEL-d} if and only if the potential satisfies \eqref{latter}, which is exactly condition \eqref{AAAp} in the statement. In particular, if a potential does not satisfy \eqref{AAAp}, then balls cannot be  energy minimisers for finite $m>0$.

From now on we assume that the potential $\Psi$ satisfies \eqref{AAAp}, namely that the equality in \eqref{rewritingEL-d} is satisfied, and focus on the inequality condition $\lambda\geq 0$. By \eqref{AAAb} and 
\eqref{latter} we have that 
$$
\lambda=\left(q+\frac{m}2\right)\left(\frac{m}{|B|}\right)^{2/d},
$$
and hence $\lambda\geq 0$ if and only if $m\geq -2q$. In other words, if $\Psi$ satisfies \eqref{AAAp}, the ball $B_{R_{m}}$ minimises $\mathcal I$ on $\mathcal A_m$  for every $m\geq -2q$.
Therefore, the critical threshold is given by $m^*=-2q$, which can be written as \eqref{mast-b} by \eqref{latter}.
%
%Now note that, if $\Psi$ satisfies \eqref{AAAp}, the ball $B_{R_{m}}$, for $m=-2q$, satisfies both conditions \eqref{rewritingEL-d}, with the equality $\lambda=0$. Equivalently, it satisfies \eqref{EL1E-d}, \eqref{EL2E-d} with the equality sign, and \eqref{EL3E-d}. Hence $\frac{1}{m}\chi_{B_{R_{m}}}$, for $m=-2q$, is a solution of \eqref{old-EL1-d}--\eqref{old-EL2-d}, namely it is the unique minimiser of $\mathcal I$ on $\mathcal{P}(\R^d)$, up to translations. This implies by uniqueness that 
%\begin{equation}\label{mast-b}
%m^*=-2q=\frac{\gamma_d}d \int_{\mathbb{S}^{d-1}} \Psi(y)\,d\mathcal{H}^{d-1}(y).
%\end{equation}
%In conclusion,  if $\Psi$ satisfies \eqref{AAAp}, the ball $B_{R_{m}}$ minimises $\mathcal I$ on $\mathcal A_m$ for every $m\geq m^*$.
\end{proof}

\begin{remark}[Bound on the critical mass] 
%
%\begin{remark}[Bound on $m^*$ in the two-dimensional case]\label{rem:mstarpi}
%In the case $d=2$ we have that $(a_1^\ast)^2+(a_2^\ast)^2=2$. By substituting $(a_2^\ast)^2=2-(a_1^\ast)^2$ in the equality $m^\ast=|B| a_1^*a_2^*$, we have that 
%$$
%m^*=|B| a_1^*\sqrt{2-(a_1^*)^2}.
%$$
%Since the function in the right-hand side is maximised at $a_1^*=1$, we deduce immediately the bound $m^\ast\leq |B|$. Note that the bound is attained only in the special case where the minimiser of $\mathcal{I}$ on measures is supported on a disc with $a_1^*=a_2^*=1$, while in general we have $m^\ast<|B|$.
%
%We recall that $|B|$ is the threshold for the mass for isotropic energies. Hence in the anisotropic case the critical mass $m^*$ is generally smaller than in the isotropic case, at least in the two-dimensional case.
%\end{remark}
%
We now show that the critical mass $m^*$, when it exists, satisfies the bound 
\begin{equation}\label{bound-mass}
m^*\leq \frac{\gamma_d}d \int_{\mathbb{S}^{d-1}} \Psi(y)\,d\mathcal{H}^{d-1}(y),
\end{equation}
where the equality is attained when the support of the minimisers of $\mathcal I$ on $\mathcal P(\R^d)$ is a ball. %We recall that this bound was proved for $d=2$ in Remark~\ref{rem:mstarpi}.

First note that the right-hand side of \eqref{bound-mass} is exactly given by \eqref{mast-b}, namely by the value of the critical mass if minimisers of $\mathcal I$ on measures are supported on a ball.  It remains to show that the critical mass $m^*$ satisfies the inequality in \eqref{bound-mass} in the general case. %It is sufficient to prove the statement in the case where $E^*$ is a non-degenerate ellipsoid.

By \eqref{old-EL1-d}, if the minimiser of $\mathcal I$ on measures, with barycentre at $0$, is of the form $\mu^*=\frac{1}{m^*}\chi_{E^*}$, it satisfies the condition
$$
\partial_{ij}(W\ast \chi_{E^\ast})(x)+m^\ast \delta_{ij} = 0, \quad i, j=1,\dots,d,
$$
for every $x\in E^*$, which thanks to \eqref{WEd} can be written as
\begin{equation}\label{EEEE}
\frac{\gamma_d}{|B|} \int_{\mathbb{S}^{d-1}} \frac{\Psi(y)\, y_i y_j}{(M^*y\cdot y)^{d/2}}\,d\mathcal{H}^{d-1}(y)=\delta_{ij},  \quad i,j=1,\dots,d,
\end{equation}
where $\gamma_d$ is defined in \eqref{def-gamma} and $M^*\in \mathbb{M}_+$ is the matrix associated with $E^*$ (as in Step~2 of the proof of Theorem~\ref{Mainthm}) . From \eqref{EEEE} one can immediately obtain the equalities
\begin{align}\label{uno}
d&=\frac{\gamma_d}{|B|} \int_{\mathbb{S}^{d-1}} \frac{\Psi(y)}{(M^*y\cdot y)^{\frac{d}2}}\,d\mathcal{H}^{d-1}(y),\\\label{due}
\mathrm{tr}(M^*) &= \frac{\gamma_d}{|B|} \int_{\mathbb{S}^{d-1}} \frac{\Psi(y)}{(M^*y\cdot y)^{\frac{d}2-1}}\,d\mathcal{H}^{d-1}(y),
\end{align}
by adding up the diagonal terms in \eqref{EEEE}, and by multiplying the $ij$-th term by $M^*_{ij}$ and then adding the resulting terms up, respectively. 
Since
$$
\mathrm{tr}(M^*) =\sum_{i=1}^d (a^*_i)^2 
$$
where $a_i^*$ are the semi-axes of $E^*$, we conclude that  
\begin{align*}
\sum_{i=1}^d (a^*_i)^2 &\leq \frac{\gamma_d}{|B|} \left(\int_{\mathbb{S}^{d-1}} \frac{\Psi(y)}{(M^*y\cdot y)^{\frac{d}2}}\,d\mathcal{H}^{d-1}(y)\right)^{\frac{d-2}d}
\left(\int_{\mathbb{S}^{d-1}} \Psi(y)\,d\mathcal{H}^{d-1}(y)\right)^{\frac2d}\\
&=\frac{\gamma_d}{|B|} \left(\frac{d|B|}{\gamma_d} \right)^{\frac{d-2}d}\left(\int_{\mathbb{S}^{d-1}} \Psi(y)\,d\mathcal{H}^{d-1}(y)\right)^{\frac2d},
\end{align*}
where we have used \eqref{uno}--\eqref{due}, and H\"older's inequality. Recalling that 
$$
\left(\prod_{i=1}^d (a_i^*)^2\right)^{1/d}\leq \frac1d \sum_{i=1}^d (a_i^*)^2,
$$
we obtain the bound 
\begin{align*}
\left(\prod_{i=1}^d (a_i^*)^2\right)^{1/d}&\leq \frac{1}{d}\frac{\gamma_d}{|B|} \left(\frac{d|B|}{\gamma_d} \right)^{\frac{d-2}d}\left(\int_{\mathbb{S}^{d-1}} \Psi(y)\,d\mathcal{H}^{d-1}(y)\right)^{\frac2d}\\
&=\left( \frac{1}{d}\frac{\gamma_d}{|B|}\int_{\mathbb{S}^{d-1}} \Psi(y)\,d\mathcal{H}^{d-1}(y)\right)^{\frac2d}.
\end{align*}
Hence, 
$$
m^* = |B| \prod_{i=1}^d a_i^* \leq \frac{\gamma_d}{d}\int_{\mathbb{S}^{d-1}} \Psi(y)\,d\mathcal{H}^{d-1}(y),
$$
as claimed.

%In Remark \ref{rem:mstarpi} we have shown that for $d=2$ the critical mass $m^*$ satisfies the optimal bound $m^*\leq \pi$, which is only attained when the support $E^*$ of the minimiser $\mu^*$ of $\mathcal I$ is a disc. This is in agreement with formula \eqref{mast-b} above, which shows that for $d=2$, by \eqref{average-Psi} and \eqref{def-gamma}, $m^*=- 2q=\pi$ for a disc, and $R_{m^*}=1$.
%
%{\color{red} If $d=3$, by \eqref{mast-b} and \eqref{def-gamma}, we have that if $E^*$ is a ball, then 
%$$
%m^* = -2q = \frac{1}{3}\sqrt{\frac{\pi}2} \int_{\mathbb{S}^{2}} \Psi(y)\,d\mathcal{H}^{2}(y),
%$$
%and the minimal ball has radius 
%$$
%R_{m^*}=\left(\frac{m^*}{|B|}\right)^{\frac13} = 
%\left(\frac{1}{3|B|}\sqrt{\frac{\pi}2} \int_{\mathbb{S}^{2}} \Psi(y)\,d\mathcal{H}^{2}(y)\right)^{\frac13} .
%$$
%We now show that a general $E^*$ satisfies the bound 
%$$
%m^*\leq \frac{1}{3}\sqrt{\frac{\pi}2} \int_{\mathbb{S}^{2}} \Psi(y)\,d\mathcal{H}^{2}(y),
%$$
%as for the case $d=2$.
%} 
\end{remark}

%\subsection{Some properties of the energy}\label{sec:prop-en}

In the next lemma we show that, up to a multiplicative factor and an additive constant depending on the dimension $d$ and on the potential $W$, 
the auxiliary function $f$ defined in \eqref{def-fg}--\eqref{def-f-d} coincides with the nonlocal energy $\mathcal{I}$ in \eqref{en:usual} evaluated on normalised characteristic functions of ellipsoids.
In particular, the proof of Theorem~\ref{Mainthm} shows that, for an ellipsoid, being a constrained critical point of $f$ on the set $\mathcal M$ is in fact equivalent to being a critical point of $\mathcal{I}$ on densities. In this identification the set $\mathcal M$ corresponds to the set of densities of the form $\chi_E/|E|$ for $E$ ellipsoid centred at $0$ with $|E|\geq m$.

\begin{lemma}\label{lemma}
Let $f:\mathbb{M}_+\to \R$ be defined as in \eqref{def-fg}--\eqref{def-f-d}. 
Then
\begin{equation}\label{eq:fIm}
\frac{1}{d+2} f(M)= \mathcal I\left(\frac{\chi_E}{|E|}\right) + c_d,
\end{equation}
where $c_d$ is a constant depending on the dimension $d$ and on the potential $W$, and $E\subset \R^d$ is the ellipsoid defined as 
$$
E:=\big\{x\in \R^d: \ Mx\cdot x\leq 1\big\}.
$$
\end{lemma}
\begin{proof}
To prove \eqref{eq:fIm} we start working with the quadratic term in $\mathcal{I}$. Let $x\in E$; by the change of variables $x=M^{1/2}z$, with $z\in B$, and since $\det M^{1/2}=\frac{|E|}{|B|}$, we have 
\begin{align}
\frac1{|E|}\int_E|x|^2 \,dx = \frac{1}{|B|}\int _B Mz\cdot z \, dz =  \frac{1}{|B|} \text{tr}(M) \int_{B}z_1^2 \,dz,
\end{align}
where we have used the symmetries of $B$. Since
$$
\int_{B}z_1^2 \, dz = \frac1{d+2}\, |B|,
$$
we conclude that 
\begin{equation}\label{trace:term}
\frac1{|E|^2}\int_E\int_E \frac12 |x-y|^2 \,dxdy =  \frac 1{d+2} \text{tr}(M).
\end{equation}
We now work on the repulsive term in $\mathcal{I}$. We proceed differently for $d=3$ and $d=2$. 

Let $d=3$. By integrating the expression \eqref{pot-dd} in $E$, and by using \eqref{quad:pot} and \eqref{AAc-d}, we have
\begin{align}\label{quad:int}
\int_{E} (W\ast \chi_E)(x) \,dx &= \int_E Qx\cdot x \, dx + |E|\,c 
= \frac{|E|}{|B|} \int_B (M^{1/2} Q M^{1/2}) z\cdot z\, dz + |E|\,c \nonumber\\
&= \frac{|E|}{|B|}\frac{|B|}{5}\, \text{tr}(M^{1/2} Q M^{1/2}) + |E|\,c  =  \frac{|E|}{5} Q\cdot M + |E|\,c\nonumber\\
&= - \frac{|E|^2}{10 |B|}\gamma_3 \int_{\mathbb{S}^{2}} \frac{\Psi(y)}{\sqrt{My\cdot y}}\,d\mathcal{H}^{2}(y)+|E|\,c.
\end{align}
By  \eqref{pot-dd} and \eqref{Da-M} we have
$$
c= \frac{|E|}{2|B|}\gamma_3
\int_{\mathbb{S}^2} \frac{\Psi(y)}{\sqrt{My\cdot y}}\,d\mathcal{H}^2(y).
$$%
Hence \eqref{quad:int} gives 
\begin{equation}\label{nonlocal-term}
\frac1{|E|^2}\int_{E} (W\ast \chi_E)(x)\, dx 
= \frac{2}{5}\frac{\gamma_3}{|B|}
\int_{\mathbb{S}^2} \frac{\Psi(y)}{\sqrt{My\cdot y}}\,d\mathcal{H}^2(y).
\end{equation}
In conclusion, \eqref{trace:term} and \eqref{nonlocal-term} give \eqref{eq:fIm} for $d=3$ with $c_d=0$. 

Let $d=2$. In \cite[Lemma~7.1]{CS2d} Carrillo and Shu proved that in 2d a kernel $W$ of the form \eqref{potdef-dd} can be written as
$$
W(x)=-\frac1{2\pi}\int_{\mathbb{S}^1}\Psi(y)\log |x\cdot y|d\mathcal{H}^1(y) +c
$$
for a suitable constant $c\in\R$ (depending only on $W$), where $\Psi$ is the function in \eqref{FTW-dd}. Therefore, 
\begin{equation}\label{CS-formula}
\int_{E} (W\ast \chi_E)(x) \,dx 
= -\frac1{2\pi}\int_{\mathbb{S}^1}\Psi(y)\left(\iint_{E\times E}\log |(x-z)\cdot y|dxdz \right)\, d\mathcal{H}^1(y) +|E|^2 c.
\end{equation}
To compute the integral at the right-hand side we recall two facts. Let $y\in\mathbb{S}^1$ and let $y^\perp:=(-y_2,y_1)\in \mathbb{S}^1$.
A simple computation shows that the projection of $\chi_E$ on the line passing through $0$ of direction $y$ is a semicircle law of parameter $r(y):=\sqrt{My\cdot y}$, that is,
\begin{equation}\label{fact1}
\int_\R \chi_E(ty+sy^\perp)\,ds = \frac{2|E|}{\pi}\frac1{r^2(y)}\sqrt{r^2(y)-t^2}\chi_{[-r(y),r(y)]}(t),
\end{equation}
for every $t\in\R$. Moreover, the 1d logarithmic potential of a semicircle law is quadratic in the support of the semicircle law; more precisely,
for every $r>0$ and $\xi\in(-r,r)$
\begin{equation}\label{fact2}
-\frac2{\pi}\frac1{r^2}\int_{-r}^r\log |\xi-t| \sqrt{r^2-t^2}\,dt=-\frac1{r^2}\xi^2-\log\frac{r}2+\frac12.
\end{equation}
Now, for any $x,z\in E$ we write $x=x_1y+x_2y^\perp$
and $z=z_1y+z_2y^\perp$, so that by \eqref{fact1} and \eqref{fact2} we have
\begin{align*}
-&\frac1{2\pi}\iint_{E\times E}\log |(x-z)\cdot y|\,dxdz \\
&= -\frac1{2\pi}\iint_{\R\times \R}\log |x_1-z_1| \left( \int_\R \chi_E(x_1y+x_2y^\perp)dx_2\right)
\left( \int_\R \chi_E(z_1y+z_2y^\perp)dz_2\right)
dx_1dz_1 \\
& = \frac{|E|^2}{\pi^2}\int_{-r(y)}^{r(y)}\frac1{r^2(y)}\sqrt{r^2(y)-x_1^2}\left(-\frac1{r^2(y)}x_1^2-\log\frac{r(y)}{2} +\frac12\right)dx_1
\\
& = \frac{|E|^2}{\pi^2}\int_{-1}^{1}\sqrt{1-t^2}\left(-t^2-\log\frac{r(y)}{2}+\frac12\right)dt
\\
&= -\frac{|E|^2}{2\pi} \log r(y)+|E|^2\tilde c,
\end{align*}
where $\tilde c$ is a universal constant. Recalling that $r(y)=\sqrt{My\cdot y}$ by definition, we conclude by \eqref{CS-formula} that
\begin{equation}\label{nonlocal-term2}
\frac1{|E|^2}\int_{E} (W\ast \chi_E)(x) \,dx 
= -\frac{1}{4\pi}\int_{\mathbb{S}^1}\Psi(y)\log (My\cdot y)\,d\mathcal{H}^1(y) + \bar c
\end{equation}
for some constant $\bar c$ depending only on $W$.
In conclusion, \eqref{trace:term} and \eqref{nonlocal-term2} give \eqref{eq:fIm} for $d=2$. 
\end{proof}

\noindent
\textbf{Acknowledgements.}
MGM is a member of GNAMPA--INdAM. MGM acknowledges support by MUR under the grant PRIN2022J4FYNJ. 
LS acknowledges support by the EPSRC under the grants EP/V00204X/1 and EP/V008897/1.

%%%%%%%%%%%%%%%%%%%%%%%%%%%%%%%%%%%%%%%%%%%%%%%%%

\end{document}